\documentclass{article}
\usepackage{amsmath}
\usepackage{amssymb}
\usepackage{amsthm}
\usepackage{cite}

\renewcommand{\vec}[1]{\boldsymbol{#1}}
\newcommand{\+}[1]{\vec{#1}}

\newcommand{\lt}{\left}
\newcommand{\rt}{\right}
\newcommand\norm[1]{\left\lVert#1\right\rVert}
\newcommand{\rvec}[1]{\lt(#1\rt)}
\newcommand{\pmat}[1]{\begin{pmatrix}#1\end{pmatrix}}
\newcommand{\bmat}[1]{\begin{bmatrix}#1\end{bmatrix}}

\DeclareMathOperator{\diag}{diag}

\newtheorem{algorithm}{Algorithm}
\newtheorem{theorem}{Theorem}
\newtheorem{lemma}{Lemma}
\newtheorem{remark}{Remark}

\newcommand{\x}{\vec{x}}
\newcommand{\y}{\vec{y}}
\newcommand{\z}{\vec{z}}
\newcommand{\e}{\vec{e}}
\newcommand{\w}{\vec{w}}
\newcommand{\0}{\vec{0}}
\newcommand{\dx}{\vec{\Delta x}}
\newcommand{\dy}{\vec{\Delta y}}
\newcommand{\dz}{\vec{\Delta z}}
\newcommand{\xk}{\vec{x}^k}
\newcommand{\yk}{\vec{y}^k}
\newcommand{\zk}{\vec{z}^k}
\newcommand{\du}{\vec{\Delta u}}
\newcommand{\dv}{\vec{\Delta v}}
\newcommand{\wmin}{w_{\min}}
\newcommand{\eps}{\varepsilon}

\begin{document}
\title{An Inexact Potential Reduction Method for Linear Programming}
\date{\vspace{-5ex}}
\author{Lukas Schork\thanks{L.Schork@ed.ac.uk} \and Jacek
  Gondzio\thanks{J.Gondzio@ed.ac.uk}}
\maketitle

\begin{center}
\textit{School of Mathematics, University of Edinburgh, Edinburgh EH9 3FD,
  Scotland, UK} \\[5ex]
\textbf{Technical Report, July 2016}
\end{center}
\vspace{3ex}

\begin{abstract}
A class of interior point methods using inexact directions is analysed. The
linear system arising in interior point methods for linear programming is
reformulated such that the solution is less sensitive to perturbations in the
right-hand side. For the new system an implementable condition is formulated
that controls the relative error in the solution. Based on this condition, a
feasible and an infeasible potential reduction method are described which retain
the convergence and complexity bounds known for exact directions.
\end{abstract}

\section{Introduction}
\label{sec:introduction}

The primal-dual interior point method (IPM) is one of the most widely used
methods for solving large linear programming problems. The method can be
analysed and implemented as a path-following algorithm, in which the iterates
follow a central trajectory toward the solution set, or as a potential reduction
algorithm, which makes progress by systematically reducing a potential function.
Most implementations make use of the path-following concept \cite{andersen1996}.

This paper analyses a variant of the IPM that works with inexactly computed step
directions. Inexact directions occur in implementations which solve the linear
equation systems by iterative methods. The analysis given here is closely
related to such an implementation and provides criteria to control the level of
inexactness in the computation.

The paper introduces two interior point algorithms that work with inexact
directions. The first algorithm requires a strictly feasible starting point and
keeps all iterates feasible. The second algorithm can start from an infeasible
point and achieves feasibility in the limit. Both algorithms are formulated and
analysed as potential reduction methods. It is proved that in both cases the
\emph{inexact} methods retain the convergence and complexity bounds of the
\emph{exact} ones.

The linear program is stated in standard form of a primal-dual pair
\begin{alignat}{2}
  \label{eq:primal}
  &\text{minimize }\+c^T\x &\quad& \text{subject to }A\x=\+b,\;\x\ge\0, \\
  \label{eq:dual}
  &\text{maximize }\+b^T\y && \text{subject to }A^T\y+\z=\+c,\;\z\ge\0,
\end{alignat}
in which $A$ is an $m\times n$ matrix of full row rank. An IPM generates a
sequence of iterates $\rvec{\xk,\yk,\zk}$ by taking steps along the Newton
direction to the nonlinear system
\begin{equation}
  \label{eq:nonlinear}
  F(\x,\y,\z) := \pmat{A\x-\+b \\ A^T\y+\z-\+c \\ X\z-\mu\e} = \0,
\end{equation}
in which $X:=\diag(\x)$, $\e$ is the $n$-vector of ones and $\mu>0$ is a
parameter that is gradually reduced to zero. The step directions are computed
from the linear system
\begin{equation}
  \label{eq:newton}
  \bmat{A&0&0 \\ 0&A^T&I \\ Z^k&0&X^k}
  \pmat{\dx^* \\ \dy^* \\ \dz^*} =
  \pmat{\+b-A\xk \\ \+c-A^T\yk-\zk \\ -X^k\zk+\mu\e},
\end{equation}
in which $X^k:=\diag(\xk)$ and $Z^k:=\diag(\zk)$. The step sizes are chosen to
keep $\xk$ and $\zk$ positive.

The potential reduction method is a particular instance of the IPM. It sets
$\mu=(\xk)^T\zk/(n+\nu)$ for a constant $\nu\ge\sqrt{n}$ and chooses a step size
to decrease a potential function by at least a certain constant. This paper uses
the Tanabe-Todd-Ye potential function \cite{tanabe1988,todd1990}
\begin{equation*}
  \phi(\x,\z) := (n+\nu)\ln(\x^T\z) - \sum_{i=1}^n\ln(x_iz_i) -n\ln n.
\end{equation*}

The inexact methods work with step directions of the form
\begin{equation}
  \label{eq:inewton}
  \bmat{A&0&0 \\ 0&A^T&I \\ Z^k&0&X^k}
  \pmat{\dx \\ \dy \\ \dz} =
  \pmat{\+b-A\xk \\ \+c-A^T\yk-\zk \\ -X^k\zk+\mu\e+\+\xi_0},
\end{equation}
in which a residual $\+\xi_0$ remains in the complementarity equations. The
primal and dual feasibility equations must be satisfied exactly. Conditions will
be imposed on $\+\xi_0$ to guarantee that the step decreases $\phi$
sufficiently.

\section{The Inexact Potential Reduction Method}
\label{sec:algorithm}

Considering one iterate $\rvec{\xk,\yk,\zk}$, we define diagonal matrices
\begin{equation*}
  D:=(X^k)^{1/2}(Z^k)^{-1/2}, \quad W:=(X^kZ^k)^{1/2},
\end{equation*}
and $\w:=W\e$. To analyse the step directions it is convenient to write the
Newton system \eqref{eq:newton} in the scaled quantities $\du^*:=D^{-1}\dx^*$
and $\dv^*:=D\dz^*$, which is
\begin{equation}
  \label{eq:snewton}
  \bmat{AD&0&0 \\ 0&DA^T&I \\ I&0&I}
  \pmat{\du^* \\ \dy^* \\ \dv^*} =
  \pmat{\+b-A\xk \\ D(\+c-A^T\yk-\zk) \\ -\w+\mu W^{-1}\e} =:
  \pmat{\+p \\ \+q \\ \+r}.
\end{equation}
The inexact solution corresponding to a residual $\+\xi$ in the scaled system
then satisfies
\begin{equation}
  \label{eq:isnewton}
  \bmat{AD&0&0 \\ 0&DA^T&I \\ I&0&I}
  \pmat{\du \\ \dy \\ \dv} =
  \pmat{\+p \\ \+q \\ \+r+\+\xi}.
\end{equation}

The inexact potential reduction algorithms make use of the following conditions
on the residual, in which $\kappa\in[0,1)$ and $\norm{\cdot}$ is the Euclidean
norm:
\begin{subequations}
  \begin{align}
    \label{eq:cond1}
    -\+r^T\+\xi &\le \kappa \norm{\+r}^2, \\
    \label{eq:cond2}
    \norm{\+\xi} &\le \kappa \min(\norm{\du},\norm{\dv}), \\
    \label{eq:cond3}
    -\w^T\+\xi &\le \kappa n/(n+\nu) \norm{\w}^2.
  \end{align}
\end{subequations}

Algorithm~\ref{alg:feasible} is the inexact version of the feasible potential
reduction method described in \cite{kojima1991,wright1997}. All iterates belong
to the strictly feasible set
\begin{equation*}
  \mathcal{F}^o := \left\{(\x,\y,\z): A\x=\+b, A^T\y+\z=\+c, (\x,\z)>0\right\},
\end{equation*}
which is assumed to be nonempty. The algorithm does not require condition
\eqref{eq:cond3}.

\begin{algorithm}
  \label{alg:feasible}
  Given $\rvec{\x^0,\y^0,\z^0}\in\mathcal{F}^o$ and $\eps>0$. Choose
  $\nu\ge\sqrt{n}$ and $\kappa\in[0,1)$. Set $\delta:=0.15(1-\kappa)^4$ and
  $k:=0$.
  \begin{enumerate}
    \item If $(\xk)^T\zk\le\eps$ then stop.
    \item Let $\mu:=(\xk)^T\zk/(n+\nu)$. Compute the solution to
      \eqref{eq:isnewton} with residual $\+\xi$ that satisfies
      \eqref{eq:cond1}--\eqref{eq:cond2}. Set $\dx:=D\du$ and $\dz:=D^{-1}\dv$.
    \item Find step size $\alpha^k$ such that
      \begin{equation}
        \label{eq:stepsize1}
        \phi(\xk+\alpha^k\dx, \zk+\alpha^k\dz) \le \phi(\xk,\zk)-\delta.
      \end{equation}
    \item Set $\rvec{\x^{k+1},\y^{k+1},\z^{k+1}} := \rvec{\xk,\yk,\zk} +
      \alpha^k\rvec{\dx,\dy,\dz}$, $k:=k+1$ and go to 1.
  \end{enumerate}
\end{algorithm}

The following theorem, which is proved in Section~\ref{sec:proof1}, states
that Algorithm~\ref{alg:feasible} retains the complexity bound of the exact
version analysed in \cite{kojima1991,wright1997}.

\begin{theorem}
  \label{thm:feasible}
  Let $\rvec{\x^0,\y^0,\z^0}\in\mathcal{F}^o$ and $L\ge0$ such that
  $\phi(\x^0,\z^0)=O(\nu L)$. Suppose that $\ln(1/\eps)=O(L)$. Then
  Algorithm~\ref{alg:feasible} terminates in $O(\nu L)$ iterations provided that
  $\kappa$ is chosen independently of $n$.
\end{theorem}

Algorithm~\ref{alg:infeasible} is an infeasible inexact potential reduction
method, as its sequence of iterates does not, in general, belong to
$\mathcal{F}^o$. It extends Algorithm~1 from \cite{mizuno1995} to work with
inexact directions. Given positive constants $\rho$ and $\eps$, it finds
$\eps$-accurate approximations to solutions $\x^*$ to \eqref{eq:primal} and
$\rvec{\y^*,\z^*}$ to \eqref{eq:dual}, if they exist, such that
\begin{equation*}
  \norm{\rvec{\x^*,\z^*}}_\infty \le \rho.
\end{equation*}

\begin{algorithm}
  \label{alg:infeasible}
  Given $\rho>0$ and $\eps>0$. Set
  $\rvec{\x^0,\y^0,\z^0}=\rho\rvec{\e,\0,\e}$.
  Choose $\sqrt{n}\le\nu\le2n$ and $\kappa\in[0,1)$. Set
  $\delta:=(1-\kappa)^4/(1600(n+\nu)^2)$ and $k:=0$.
  \begin{enumerate}
    \item If $(\xk)^T\zk\le\eps$ then stop.
    \item Let $\mu:=(\xk)^T\zk/(n+\nu)$. Compute the solution to
      \eqref{eq:isnewton} with residual $\+\xi$ that satisfies
      \eqref{eq:cond1}--\eqref{eq:cond3}. Set $\dx:=D\du$ and $\dz:=D^{-1}\dv$.
    \item Find step size $\alpha^k$ such that
      \begin{subequations}
      \begin{align}
        \label{eq:step2a}
        \phi(\xk+\alpha^k\dx, \zk+\alpha^k\dz) &\le \phi(\xk,\zk)-\delta, \\
        \label{eq:step2b}
        (\xk+\alpha^k\dx)^T(\zk+\alpha^k\dz) &\ge (1-\alpha^k) (\xk)^T\zk.
      \end{align}
      \end{subequations}
      If no such step size exists then stop.
    \item Set $\rvec{\x^{k+1},\y^{k+1},\z^{k+1}} := \rvec{\xk,\yk,\zk} +
      \alpha^k\rvec{\dx,\dy,\dz}$, $k:=k+1$ and go to 1.
  \end{enumerate}
\end{algorithm}

The following theorem, which is proved in Section~\ref{sec:proof2}, states that
Algorithm~\ref{alg:infeasible} retains the complexity bound of the exact
infeasible potential reduction method \cite{mizuno1995}.

\begin{theorem}
  \label{thm:infeasible}
  Let $L\ge\ln n$ such that $\rho=O(L)$. Suppose that $\ln(1/\eps)=O(L)$. Then
  Algorithm~\ref{alg:infeasible} terminates in $O(\nu(n+\nu)^2L)$ iterations
  provided that $\kappa$ is chosen independently of $n$. If the algorithm stops
  in step 1 then the iterate is an $\eps$-approximate solution; otherwise it
  stops in step 3 showing that there are no optimal solutions $\x^*$ to
  \eqref{eq:primal} and $\rvec{\y^*,\z^*}$ to \eqref{eq:dual} such that
  $\norm{\rvec{\x^*,\z^*}}_\infty\le\rho$.
\end{theorem}

The following lemma is key to the analysis of the inexact potential reduction
methods given in the next two sections. It exploits the particular form of the
scaled Newton system to prove that condition \eqref{eq:cond2} bounds the
\emph{relative error} in the inexact solution.
\begin{lemma}
  \label{lem:relerr}
  Given solutions to \eqref{eq:snewton} and \eqref{eq:isnewton}, suppose that
  \eqref{eq:cond2} holds for $\kappa\in[0,1)$. Then
  \begin{equation*}
    \frac{\norm{\du-\du^*}}{\norm{\du^*}} \le \frac{\kappa}{1-\kappa}, \quad
    \frac{\norm{\dv-\dv^*}}{\norm{\dv^*}} \le \frac{\kappa}{1-\kappa}.
  \end{equation*}
\end{lemma}
\begin{proof}
  Denoting $P:=DA^T(AD^2A^T)^{-1}AD$, the solution to \eqref{eq:snewton} is
  \begin{subequations}
    \begin{align*}
      \du^* &= DA^T(AD^2A^T)^{-1}\+p - (I-P)\+q + (I-P)\+r, \\
      \dy^* &= (AD^2A^T)^{-1}(\+p + AD\+q - AD\+r), \\
      \dv^* &= -DA^T(AD^2A^T)^{-1}\+p + (I-P)\+q + P\+r.
    \end{align*}
  \end{subequations}
  It follows that
  \begin{align*}
    \du-\du^* &= (I-P)\+\xi, \\
    \dv-\dv^* &= P\+\xi.
  \end{align*}
  Because $P$ and $(I-P)$ are projection operators, $\norm{P}\le1$ and
  $\norm{(I-P)}\le1$. Therefore the absolute errors are bounded by the norm of
  the residual,
  \begin{align*}
    \norm{\du-\du^*} &\le \norm{\+\xi}, \\
    \norm{\dv-\dv^*} &\le \norm{\+\xi}.
  \end{align*}
  On the other hand, it follows from the triangle inequality and
  \eqref{eq:cond2} that
  \begin{subequations}
  \begin{align}
    \label{eq:ubound}
    \norm{\du^*} &= \norm{\du-(I-P)\+\xi} \ge \norm{\du}-\norm{\+\xi}
    \ge (1-\kappa)\norm{\du}, \\
    \label{eq:vbound}
    \norm{\dv^*} &= \norm{\dv-P\+\xi} \ge \norm{\dv}-\norm{\+\xi}
    \ge (1-\kappa)\norm{\dv}.
  \end{align}
  \end{subequations}
  Combining both inequalities and \eqref{eq:cond2} gives
  \begin{align*}
    \frac{\norm{\du-\du^*}}{\norm{\du^*}} &\le \frac{\norm{\+\xi}}{\norm{\du^*}}
    \le \frac{\kappa\norm{\du}}{(1-\kappa)\norm{\du}} = \frac{\kappa}{1-\kappa},
    \\
    \frac{\norm{\dv-\dv^*}}{\norm{\dv^*}} &\le \frac{\norm{\+\xi}}{\norm{\dv^*}}
    \le \frac{\kappa\norm{\dv}}{(1-\kappa)\norm{\dv}} = \frac{\kappa}{1-\kappa}
  \end{align*}
  as claimed.
\end{proof}

\section{Proof of Theorem~\ref{thm:feasible}}
\label{sec:proof1}

This and the next section use two technical results from Mizuno, Kojima and Todd
\cite{mizuno1995}, which are stated in the following two lemmas.

\begin{lemma}
  \label{lem:qbound}
  For any $n$-vectors $\x>0$, $\z>0$, $\dx$, $\dz$ and $\alpha>0$ such that
  $\norm{\alpha X^{-1}\dx}_\infty \le \tau$ and $\norm{\alpha Z^{-1}\dz}_\infty
  \le \tau$ for a constant $\tau\in(0,1)$ it holds true that
  \begin{equation*}
    \phi(\x+\alpha\dx,\z+\alpha\dz) \le \phi(\x,\z) + g_1\alpha + g_2\alpha^2
  \end{equation*}
  with coefficients
  \begin{align*}
    g_1 &= \lt(\frac{n+\nu}{\x^T\z}\e-(XZ)^{-1}\e\rt)^T (Z\dx+X\dz), \\
    g_2 &= (n+\nu)\frac{\dx^T\dz}{\x^T\z} +
    \frac{\norm{X^{-1}\dx}^2+\norm{Z^{-1}\dz}^2}{2(1-\tau)}.
  \end{align*}
\end{lemma}

\begin{lemma}
  \label{lem:wbound}
  For any $n$-vector $\w>0$ and $\nu\ge\sqrt{n}$
  \begin{equation*}
    \lt\lVert W^{-1}\e-\frac{n+\nu}{\w^T\w}\w\rt\rVert \ge
    \frac{\sqrt{3}}{2\wmin},
  \end{equation*}
  where $W:=\diag(\w)$ and $\wmin:=\min_i w_i$.
\end{lemma}

Applying Lemma~\ref{lem:wbound} to the vector $\+r$ defined in
\eqref{eq:snewton} shows that
\begin{equation}
  \label{eq:rbound}
  \norm{\+r}=\norm{-\w+\mu W^{-1}\e}=\mu\norm{-\frac{1}{\mu}\w+W^{-1}\e}
  \ge \mu\frac{\sqrt{3}}{2\wmin}.
\end{equation}

The following lemma extends the analysis of the feasible potential reduction
method given in \cite{wright1997}. It shows that Algorithm~\ref{alg:feasible}
finds a step size that reduces $\phi$ by at least the prescribed value in each
iteration.

\begin{lemma}
  \label{lem:deltaphi1}
  In the $k$-th iteration of Algorithm~\ref{alg:feasible} \eqref{eq:stepsize1}
  holds for
  \begin{equation*}
    \alpha:=\frac{\wmin}{2\norm{\+r}}(1-\kappa)^3,
  \end{equation*}
  where $\wmin:=\min_i\sqrt{x_i^kz_i^k}$.
\end{lemma}
\begin{proof}
  It follows from the first two block equations in \eqref{eq:isnewton} and
  $\+p=\0$, $\+q=\0$ that
  \begin{equation*}
    \du^T\dv = -\du^TDA^T\dy = -(AD\du)^T\dy = 0,
  \end{equation*}
  and analogously $(\du^*)^T\dv^*=0$ from \eqref{eq:snewton}. Therefore
  $\norm{\du^*}^2+\norm{\dv^*}^2=\norm{\+r}^2$ and from
  \eqref{eq:ubound}, \eqref{eq:vbound} and the definition of $\alpha$
  \begin{align*}
    \norm{\alpha X^{-1}\dx}_\infty &\le \alpha\norm{W^{-1}}\norm{\du}
    \le \frac{\alpha}{\wmin} \frac{\norm{\du^*}}{1-\kappa}
    \le \frac{\alpha}{\wmin} \frac{\norm{\+r}}{1-\kappa} \le \frac{1}{2}, \\
    \norm{\alpha Z^{-1}\dz}_\infty &\le \alpha\norm{W^{-1}}\norm{\dv}
    \le \frac{\alpha}{\wmin} \frac{\norm{\dv^*}}{1-\kappa}
    \le \frac{\alpha}{\wmin} \frac{\norm{\+r}}{1-\kappa} \le \frac{1}{2}.
  \end{align*}
  Therefore $\tau:=1/2$ satisfies the assumptions of Lemma~\ref{lem:qbound}, so
  that
  \begin{equation*}
    \phi(\xk+\alpha\dx,\zk+\alpha\dz) - \phi(\xk,\zk)
    \le g_1\alpha + g_2\alpha^2
  \end{equation*}
  with coefficients
  \begin{align*}
    g_1 &= \lt( \frac{n+\nu}{\w^T\w}\e-W^{-2}\e \rt)^T W(\du+\dv) \\
    g_2 &= \norm{W^{-1}\du}^2+\norm{W^{-1}\dv}^2.
  \end{align*}
  To show that $\phi$ is sufficiently reduced along the direction
  $\rvec{\dx,\dz}$ it is necessary to show that $g_1$ is negative and bounded
  away from zero, while $g_2$ is bounded. From the definition of $\+r$ and
  condition \eqref{eq:cond1} it follows that
  \begin{subequations}
  \begin{align}
    \label{eq:g1}
    g_1 &= \lt(\frac{n+\nu}{\w^T\w}\w-W^{-1}\e\rt)^T(\du+\dv)
    = -\frac{n+\nu}{\w^T\w}\+r^T(\+r+\+\xi) \\
    \label{eq:g1bound}
    &\le -(1-\kappa)\frac{n+\nu}{\w^T\w}\norm{\+r}^2.
  \end{align}
  \end{subequations}
  For the second order term it follows from \eqref{eq:ubound}, \eqref{eq:vbound}
  that
  \begin{align*}
    g_2 &= \norm{W^{-1}\du}^2+\norm{W^{-1}\dv}^2
    \le \frac{1}{\wmin^2} \lt(\norm{\du}^2+\norm{\dv}^2\rt) \\
    &\le \frac{\norm{\du^*}^2+\norm{\dv^*}^2}{\wmin^2(1-\kappa)^2}
    = \frac{\norm{\+r}^2}{\wmin^2(1-\kappa)^2}.
  \end{align*}
  Inserting the bounds on $g_1$ and $g_2$ into the quadratic form and using the
  definition of $\alpha$ gives
  \begin{align*}
    \phi(\xk&+\alpha\dx,\zk+\alpha\dz) - \phi(\xk,\zk) \\
    &\le -(1-\kappa)\frac{n+\nu}{\w^T\w}\norm{\+r}^2 \alpha
    + \frac{\norm{\+r}^2}{\wmin^2(1-\kappa)^2} \alpha^2 \\
    &= -(1-\kappa)^4\frac{n+\nu}{\+w^T\+w}\frac{\wmin}{2}\norm{\+r}
    + \frac{(1-\kappa)^4}{4}.
  \end{align*}
  Finally, using the bound on $\norm{\+r}$ from \eqref{eq:rbound} gives
  \begin{align*}
    \phi(\xk&+\alpha\dx,\zk+\alpha\dz) - \phi(\xk,\zk) \\
    &\le (1-\kappa)^4 \lt(-\frac{\sqrt{3}}{4}+\frac{1}{4}\rt) \\
    &\le -0.15(1-\kappa)^4 = -\delta
  \end{align*}
  as claimed.
\end{proof}

The proof of Theorem~\ref{thm:feasible} is immediate.
Since $\phi(\x,\z)\ge\nu\ln(\x^T\z)$, the termination condition
\begin{equation*}
  \nu \ln\lt((\xk)^T\zk\rt) \le \nu \ln(\eps)
\end{equation*}
is satisfied when
\begin{equation}
  \label{eq:phik}
  \phi(\xk,\zk) \le \phi(\x^0,\z^0)-k\delta \le \nu\ln(\eps).
\end{equation}
Since under the assumption of the theorem $\phi(\x^0,\z^0)=O(\nu L)$ and
$\ln(1/\eps)=O(L)$, and since $\delta$ is independent of $n$, \eqref{eq:phik}
holds for $k\ge K=O(\nu L)$.

\section{Proof of Theorem~\ref{thm:infeasible}}
\label{sec:proof2}

The proof of the theorem is based on Mizuno, Kojima and Todd \cite{mizuno1995}.
We define a sequence $\{\theta^k\}$ by
\begin{equation}
  \label{eq:theta}
  \theta^0:=1 \quad \text{and} \quad \theta^{k+1}:=(1-\alpha^k) \theta^k
  \text{ for } k\ge0.
\end{equation}
Since the first two block equations in \eqref{eq:nonlinear} are linear and
satisfied exactly by a full step of the algorithm
\begin{equation*}
  \rvec{A\xk-\+b,A^T\yk+\zk-\+c}=\theta^k \rvec{A\x^0-\+b,A^T\y^0+\z^0-\+c}.
\end{equation*}

The following lemma is obtained from Lemma~4 in \cite{mizuno1995} by setting
$\gamma_0=1$ and $\gamma_1=1$.

\begin{lemma}
  \label{lem:uvbound2}
  Let $\rho>0$ and suppose that
  \begin{align}
    \rvec{\x^0,\y^0,\z^0} &= \rho \rvec{\e,\0,\e}, \notag \\
    \rvec{A\xk-\+b,A^T\yk+\zk-\+c} &=
    \theta^k \rvec{A\x^0-\+b,A^T\y^0+\z^0-\+c}, \notag \\
    \label{eq:feas} (\xk)^T\zk &\ge \theta^k (\x^0)^T\z^0.
  \end{align}
  If there exist solutions $\x^*$ to \eqref{eq:primal} and $\rvec{\y^*,\z^*}$
  to \eqref{eq:dual} such that $\norm{\rvec{\x^*,\z^*}}_\infty\le\rho$ then
  the solution to \eqref{eq:snewton} at $\rvec{\xk,\yk,\zk}$ satisfies
  \begin{align*}
    \norm{\du^*} &\le \frac{5(\xk)^T\zk}{\wmin}, \\
    \norm{\dv^*} &\le \frac{5(\xk)^T\zk}{\wmin},
  \end{align*}
  where $\wmin:=\min_i\sqrt{x^k_iz^k_i}$.
\end{lemma}

The following lemma is based on Lemma~5 in \cite{mizuno1995}. It shows that when
optimal solutions to \eqref{eq:primal} and \eqref{eq:dual} exist, then
Algorithm~\ref{alg:infeasible} can find a step size in each iteration that
satisfies \eqref{eq:step2a} and \eqref{eq:step2b}.

\begin{lemma}
  \label{lem:minstep2}
  If there exist optimal solutions $\x^*$ to
  \eqref{eq:primal} and $\rvec{\y^*,\z^*}$ to \eqref{eq:dual} such that
  $\norm{\rvec{\x^*,\z^*}}_\infty\le\rho$ then \eqref{eq:step2a} and
  \eqref{eq:step2b} hold for
  \begin{equation*}
    \alpha := \frac{(1-\kappa)^3\wmin^2}{200(n+\nu)(\xk)^T\zk}
  \end{equation*}
  in the $k$-th iteration, where $\wmin:=\min_i\sqrt{x^k_iz^k_i}$.
\end{lemma}

\begin{proof}
  A simple calculation shows that by definition of $\rvec{\x^0,\z^0}$ and
  because of \eqref{eq:step2b} the assumptions of Lemma~\ref{lem:uvbound2} are
  satisfied. Combining the lemma with \eqref{eq:ubound}, \eqref{eq:vbound} shows
  that
  \begin{align*}
    \norm{\du} &\le \frac{5(\xk)^T\zk}{(1-\kappa)\wmin}, \\
    \norm{\dv} &\le \frac{5(\xk)^T\zk}{(1-\kappa)\wmin}.
  \end{align*}
  It follows that
  \begin{align*}
    \norm{\alpha X^{-1}\dx} &\le \alpha \norm{W^{-1}} \norm{\du}
    \le \alpha \frac{5(\xk)^T\zk}{(1-\kappa)\wmin^2}
    = \frac{(1-\kappa)^2}{40(n+\nu)} \le \frac{1}{40}, \\
    \norm{\alpha Z^{-1}\dz} &\le \alpha \norm{W^{-1}} \norm{\dv}
    \le \alpha \frac{5(\xk)^T\zk}{(1-\kappa)\wmin^2}
    = \frac{(1-\kappa)^2}{40(n+\nu)} \le \frac{1}{40}.
  \end{align*}
  Therefore $\tau:=1/40$ satisfies the assumption of Lemma~\ref{lem:qbound},
  so that
  \begin{equation*}
    \phi(\xk+\alpha\dx,\zk+\alpha\dz) \le \phi(\xk,\zk) + g_1\alpha +
    g_2\alpha^2
  \end{equation*}
  with coefficients
  \begin{align*}
    g_1 &= \lt(\frac{n+\nu}{\w^T\w}\e-W^{-2}\e\rt)^T W(\du+\dv), \\
    g_2 &= \lt((n+\nu)\frac{\du^T\dv}{\w^T\w} +
    \frac{\norm{W^{-1}\du}^2+\norm{W^{-1}\dv}^2}{2(1-\tau)}\rt).
  \end{align*}

  It will be shown that $g_1$ is negative and bounded away from zero, while
  $g_2$ is bounded. Combining \eqref{eq:g1bound} and \eqref{eq:rbound} gives
  \begin{equation*}
    g_1 \le -(1-\kappa)\frac{1}{\mu}\norm{\+r}^2 \le
    -(1-\kappa)\mu\frac{3}{4\wmin^2}.
  \end{equation*}
  Next, from the bound on $\norm{\du}$ and $\norm{\dv}$ it follows that
  \begin{equation}
    \label{eq:dudv}
    |\du^T\dv| \le \norm{\du}\norm{\dv}
    \le \lt( \frac{5\w^T\w}{(1-\kappa)\wmin} \rt)^2,
  \end{equation}
  which implies that
  \begin{equation}
    \label{eq:g2a}
    (n+\nu)\frac{\du^T\dv}{\w^T\w} \le
    \frac{n+\nu}{\w^T\w}\lt( \frac{5\w^T\w}{(1-\kappa)\wmin} \rt)^2
    \le \frac{n+\nu}{n} \lt( \frac{5\w^T\w}{(1-\kappa)\wmin^2} \rt)^2,
  \end{equation}
  where the last inequality is obtained by multiplying with
  $\w^T\w/(n\wmin^2)\ge1$. Moreover, the bound on $\norm{\du}$ and
  $\norm{\dv}$ also implies that
  \begin{equation}
    \label{eq:g2b}
    \frac{\norm{W^{-1}\du}^2+\norm{W^{-1}\dv}^2}{2(1-\tau)}
    \le \frac{1}{1-\tau} \lt( \frac{5\w^T\w}{(1-\kappa)\wmin^2} \rt)^2.
  \end{equation}
  Adding up \eqref{eq:g2a} and \eqref{eq:g2b} and using $\nu\le2n$ gives
  \begin{equation*}
    g_2 \le \lt(\frac{n+\nu}{n}+\frac{1}{1-\tau}\rt) \lt(
    \frac{5\w^T\w}{(1-\kappa)\wmin^2} \rt)^2 \le
    5 \lt( \frac{5\w^T\w}{(1-\kappa)\wmin^2} \rt)^2.
  \end{equation*}
  Inserting $g_1$, $g_2$ and the definition of $\alpha$ into the quadratic form
  gives
  \begin{align*}
    \phi(\xk&+\alpha\dx,\zk+\alpha\dz) - \phi(\xk,\zk) \\
    &\le -(1-\kappa)\frac{\w^T\w}{n+\nu}\frac{3}{4\wmin^2} \alpha
    + 5 \lt( \frac{5\w^T\w}{(1-\kappa)\wmin^2} \rt)^2 \alpha^2 \\
    &= \frac{(1-\kappa)^4}{(n+\nu)^2}
    \lt(-\frac{3}{4\cdot200} + 5 \lt(\frac{5}{200}\rt)^2 \rt) = -\delta,
  \end{align*}
  which shows that $\alpha$ satisfies \eqref{eq:step2a}.

  Finally, to verify that $\alpha$ satisfies \eqref{eq:step2b}, a
  straightforward calculation shows that
  \begin{equation*}
    \dz^T\xk + \dx^T\zk = \dv^T\w + \du^T\w = \w^T(\+r+\+\xi)
    = \lt(\frac{n}{n+\nu}-1\rt)\w^T\w + \w^T\+\xi
  \end{equation*}
  and consequently
  \begin{multline*}
    (\xk+\alpha\dx)^T(\zk+\alpha\dz) = (\xk)^T\zk + \alpha(\dz^T\xk+\dx^T\zk) +
    \alpha^2 \dx^T\dz \\ = (1-\alpha)\w^T\w +
    \alpha\lt(\frac{n}{n+\nu}\w^T\w + \w^T\+\xi + \alpha\du^T\dv\rt).
  \end{multline*}
  Using \eqref{eq:dudv} and \eqref{eq:cond3} it follows for the term in
  parenthesis that
  \begin{align*}
    \frac{n}{n+\nu}\w^T\w + \w^T\+\xi + \alpha\du^T\dv &\ge
    \frac{(1-\kappa)n}{n+\nu}\w^T\w - \alpha\lt(\frac{5\w^T\w}
         {(1-\kappa)\wmin}\rt)^2 \\
    &= \frac{(1-\kappa)\w^T\w}{n+\nu}\lt(n-\frac{1}{8}\rt) > 0.
  \end{align*}
  Therefore $\alpha$ satisfies \eqref{eq:step2b}, which completes the proof.
\end{proof}

Theorem~\ref{thm:infeasible} follows from the lemma by the same argumentation as
in \cite{mizuno1995}. Under the hypothesis of the theorem
$\phi(\x^0,\z^0)=O(\nu L)$ and $\ln(1/\eps)=O(L)$. Since
$\phi(\x,\z)\ge\nu\ln(\x^T\z)$ and the potential function decreases by at
least $\delta$ in each iteration, Algorithm~\ref{alg:infeasible} terminates in
$O(\nu L/\delta)=O(\nu(n+\nu)^2L)$ iterations. When the algorithm stops in step
1, then $(\xk)^T\zk\le\eps$ and because of \eqref{eq:step2b}
\begin{equation*}
  \norm{\rvec{A\xk-\+b,A^T\yk+\zk-\+c}} \le \eps
  \norm{\rvec{A\x^0-\+b,A^T\y^0+\z^0-\+c}}/(\x^0)^T\z^0,
\end{equation*}
so that the final iterate is indeed an $\eps$-approximate solution. On the other
hand, if there exist optimal solutions $\x^*$ to \eqref{eq:primal} and
$\rvec{\y^*,\z^*}$ to \eqref{eq:dual} such that
$\norm{\rvec{\x^*,\z^*}}\le\rho$, then it follows from Lemma \ref{lem:minstep2}
that a step size exists which satisfies \eqref{eq:step2a} and \eqref{eq:step2b}.
Therefore, if the algorithm stops in step 3, then there are no such solutions.

\begin{remark}
Theorem~\ref{thm:infeasible} imposed the upper bound $\nu\le2n$, which
is not needed in the analysis of the exact potential reduction method. The
actual value of this bound, however, is not important and the proof remains
valid by adapting $\alpha$ and $\delta$ as long as $\nu=O(n)$.
\end{remark}

\section{Discussion}
\label{sec:discussion}

The analysis has shown some insights into the conditions
\eqref{eq:cond1}--\eqref{eq:cond3}. It has been seen from \eqref{eq:g1} that
$-\+r^T\+\xi<\norm{\+r}^2$ is sufficient and necessary for $\rvec{\dx,\dz}$ to
be a descent direction for $\phi$, making \eqref{eq:cond1} a necessary condition
in a potential reduction method. Condition \eqref{eq:cond2} bounds the curvature
of $\phi$ along $\rvec{\dx,\dz}$. When the iterate is feasible this condition
can be replaced by $\norm{\+r}\le c\norm{\+\xi}$ for an arbitrary constant $c$,
since then
\begin{equation*}
  \norm{\du}^2+\norm{\dv}^2=\norm{\+r+\+\xi}^2 \le (1+c)^2\norm{\+r}^2
\end{equation*}
gives the required bound on $g_2$ in Lemma~\ref{lem:deltaphi1}. For an
infeasible iterate, however, condition \eqref{eq:cond2} is needed in its form to
bound $\norm{\du}$ and $\norm{\dv}$. Finally, condition \eqref{eq:cond3}
guarantees that in the infeasible algorithm the step size restriction
\eqref{eq:step2b} can be satisfied.

Inexact directions of the form \eqref{eq:inewton} have been used and analysed in
\cite{monteiro2003,gondzio2009} in the path-following method, which sets
$\mu=\sigma\x^T\z/n$ for $\sigma<1$ and chooses the step size to keep
$x_iz_i\ge\gamma\x^T\z/n$ for a constant $\gamma\in(0,1)$. Both papers use a
basic-nonbasic splitting of the variables and solve \eqref{eq:isnewton} with
residual $\+\xi=\rvec{\+\xi_B,\+\xi_N}=\rvec{\+\xi_B,\0}$. \cite{monteiro2003}
imposes the condition
\begin{equation}
  \label{eq:monteiro}
  \norm{\+\xi_B} \le \frac{(1-\gamma)\sigma}{4\sqrt{n}}\sqrt{\x^T\z/n},
\end{equation}
whereas
\begin{equation}
  \label{eq:gondzio}
  \norm{W_B\+\xi_B}_\infty \le \eta \x^T\z/n
\end{equation}
is used in \cite{gondzio2009} with $\eta<1$ depending on $\sigma$ and $\gamma$.
Both conditions seem to require more effort by an iterative method than the
conditions used in this paper. \eqref{eq:monteiro} obviously becomes restrictive
for large problems. \eqref{eq:gondzio} is not affected by the problem dimension,
but the infinity norm does not tolerate outliers in $W_B\+\xi_B$.

Another form of inexact direction has been analysed in \cite{mizuno1999}, which
solves the complementarity equations exactly and allows a residual in the primal
and dual equations. Due to the form of the Newton system, solving the
complementarity equations exactly is trivial, whereas computing directions that
satisfy primal feasibility requires particular preconditioning techniques
\cite{gondzio2008,monteiro2003,oliveira2005}. The analysis in \cite{mizuno1999}
shows, however, that a residual in the feasibility equations must be measured in
a norm depending on $A$, which seems not to be accessible in an implementation.
Therefore this form of inexact direction is hardly useful in practice.

\bibliography{literature}{}
\bibliographystyle{plain}
\end{document}